\documentclass[11pt,authoryear]{article}
\pdfoutput = 1
\usepackage{fullpage,amsthm,amsmath,natbib,algorithm,algorithmic,enumitem,afterpage,amssymb}
\usepackage{graphicx,amsfonts,hyperref,array}
\usepackage[hang,flushmargin]{footmisc} 
\DeclareMathOperator*{\argmin}{arg\,min}
\DeclareMathOperator*{\veco}{vec}
\DeclareMathOperator*{\cov}{cov}
\DeclareMathOperator*{\tr}{tr}
\DeclareMathOperator*{\diag}{diag}

\newtheorem{theorem}{Theorem}

\newtheorem{definition}{Definition}

\begin{document}

\title{A higher-order LQ decomposition for separable covariance models}
\author{David Gerard$^1$  and Peter Hoff$^{1,2}$  \\
Departments of Statistics$^1$ and Biostatistics$^2$ \\
University of Washington}  
\maketitle

 \let\thefootnote\relax\footnotetext{Email: gerard2@uw.edu,  pdhoff@uw.edu. This research was partially supported by NI-CHD grant R01HD067509.  }

\begin{abstract}
We develop a higher order generalization of the LQ decomposition and show that this decomposition plays an important role in likelihood-based estimation and testing for separable, or Kronecker structured, covariance models, such as the multilinear normal model. This role is analogous to that of the LQ decomposition in likelihood inference for the multivariate normal model. Additionally, this higher order LQ decomposition can be used to construct an alternative version of the popular higher order singular value decomposition for tensor-valued data. We also develop a novel generalization of the polar decomposition to tensor-valued data.
\end{abstract}
\section{Introduction}
There has been a recent surge of interest in methods for tensor-valued data in the machine learning, applied math, and statistical communities. Tensors, or multiway arrays, are higher order generalizations of vectors and matrices whose elements are indexed by more than two index sets. Analysis methods for tensor-valued data include tensor decompositions and statistical modeling. The former aims to express the tensor in terms of interpretable lower-dimensional components. The latter uncovers patterns through the lens of statistical inference in a parametric statistical model.

The work in the field of tensor decompositions is extensive (see \cite{kolda2009tensor} or \cite{cichockitensor} for a review). A common class of tensor decompositions are Tucker decompositions \citep{tucker1966some}, which, for an array $X \in \mathbb{R}^{p_1\times\cdots\times p_K}$ with entries $X_{[i_1,\ldots,i_K]}$, expresses $X$ as a product of a ``core'' array $S \in \mathbb{R}^{p_1\times\cdots\times p_K}$ and matrices $U_1,\ldots,U_K$ where $U_k \in \mathbb{R}^{p_k\times p_k}$, expressed as
\begin{align}
\label{equation:tucker.decomp}
X = (U_1,\ldots,U_K) \cdot S,
\end{align}
where ``$\cdot$'' is multilinear multiplication defined in Section \ref{section:holq}. Most Tucker decompositions impose orthogonality constraints on the $U_k$'s. One resulting tensor decomposition with such orthogonality constraints is the higher order singular value decomposition (HOSVD) of \cite{de2000multilinear,de2000best}, a generalization of the singular value decomposition (SVD). There are other generalizations of the SVD to tensors outside the Tucker decomposition framework \citep{de2008tensor,grasedyck2010hierarchical,kilmer2011factorization}. However, our work will focus on Tucker decompositions of the form (\ref{equation:tucker.decomp}), where the $U_k$'s have a variety of forms other than orthogonality.

A different perspective on tensor-valued data analysis uses statistical modeling, which aims to capture the dependencies between the entries of a tensor through a parametric model. One such model is the multilinear normal model \citep{hoff2011separable,ohlson2013multilinear,manceur2013maximum} --- also known as the ``array normal model'' or ``tensor normal model'' --- which is an extension of the matrix normal model \citep{srivastava1979introduction,dawid1981some}. A $p_1 \times \cdots \times p_K$ tensor $X$ follows a multilinear normal distribution if $\veco(X)$ is normally distributed with covariance $\Sigma_K \otimes \cdots \otimes \Sigma_1$, where ``$\otimes$'' is the Kronecker product and ``$\veco(\cdot)$'' is the vectorization operator. For $\Sigma_k = A_kA_k^T$, $k = 1,\ldots,K$, the multilinear normal model may be written
\begin{align}
\label{equation:mult.norm}
X \overset{d}{=} (A_1,\ldots,A_K) \cdot Z,
\end{align}
where $Z \in \mathbb{R}^{p_1\times\cdots\times p_K}$ contains independent and identically distributed (i.i.d.) standard normal entries. The multilinear normal model ``separates'' the covariances along the modes, or dimensions of $X$. That is, the dependencies along the $k$th mode are represented by a single covariance matrix, $\Sigma_k$. Models where the covariance matrix is Kronecker structured are thus often called ``separable covariance models''. Most results for the multilinear normal model can be easily generalized to array-variate elliptically contoured models with separable covariance \citep{akdemir2011array1}.

In Section \ref{section:holq}, we derive a novel tensor decomposition, a type of Tucker decomposition, whose components provide the maximum likelihood estimators (MLEs) of the parameters in the mean zero multilinear normal model, and array-variate elliptically contoured models with separable covariance in general. This tensor decomposition is a generalization of the LQ matrix decomposition to multiway arrays, and so we call it the incredible Higher Order LQ decomposition (incredible HOLQ, or just HOLQ). One can view the LQ decomposition as taking the form
\begin{align*}
X = \ell L Q I_n\in \mathbb{R}^{p \times n},
\end{align*}
where $\ell > 0$, $Q$ has orthonormal rows, $L$ is a lower triangular matrix with positive diagonal elements and unit determinant, and $I_n$ is the identity matrix. The HOLQ takes the form
\begin{align*}
X = \ell(L_1,\ldots,L_K,I_n)\cdot Q \in\mathbb{R}^{p_1\times\cdots\times p_K \times n},
\end{align*}
where $\ell > 0$, each $L_k$ is a lower triangular matrix with positive diagonal elements and unit determinant, and $Q \in \mathbb{R}^{p_1 \times \cdots \times p_K \times n}$ has certain orthogonality properties which generalize the orthonormal rows property of the LQ decomposition. Section \ref{section:mult.norm} shows the close relationship between the HOLQ and likelihood inference in the multilinear normal model: In Section \ref{section:MLE}, we show that each $L_k$ matrix in the  HOLQ is the Cholesky square root of the MLE for the $k$th component covariance matrix, $\Sigma_k$, in the multilinear normal model (\ref{equation:mult.norm}). This relationship is analogous to the correspondence between the LQ decomposition and the MLE in the multivariate normal model. 

In the same way that likelihood estimation in the multilinear normal model is connected to the HOLQ, likelihood inference in submodels of the unconstrained multilinear normal model is connected to other decompositions where the component matrices have certain structures. In Section \ref{section:holq.junior}, we consider constraining $\Sigma_k$ to be diagonal. This has the interpretation of statistical independence along the $k$th mode and corresponds to constraining $L_k$ to be diagonal in the related tensor decomposition. We also consider constraining the diagonal of the lower triangular Cholesky square root of $\Sigma_k$ to be the vector of ones, which relates to a covariance model used in time series analysis.  We label as ``HOLQ juniors'' the class of decompositions that correspond to submodels of the unrestricted mean zero multilinear normal model. In Section \ref{section:LRT}, we use HOLQ juniors to develop a class of likelihood ratio tests for covariance models in elliptically contoured random arrays with separable covariance.

Other tensor decompositions related to the  HOLQ are discussed in Section \ref{section:other.tensor}. In Section \ref{section:isvd} we use the HOLQ to create a new higher order analogue to the SVD where each mode has singular values and vectors separated from the core array. Since this SVD is derived from the incredible HOLQ, we call it the incredible SVD (ISVD). The ISVD may be viewed as a core rotation of the HOSVD. In Section \ref{section:ihop} we use a novel minimization formulation of the polar decomposition to generalize it to tensors.
\section{The incredible HOLQ}
\label{section:holq}
Let $X \in \mathbb{R}^{p \times n}$ be of rank $p$ where $p \leq n$. Recall the LQ decomposition,
\begin{align*}
X = LQ,
\end{align*}
where $L \in G_{p}^+$, the set of $p$ by $p$ lower triangular matrices with positive diagonal elements, and $Q^T \in \mathcal{V}_{p,n}$, the Stiefel manifold of $n$ by $p$ matrices with orthonormal columns. It is common to formulate the LQ decomposition as a Gram-Schmidt orthogonalization of the rows of $X$. We instead consider an alternative formulation of the LQ decomposition as a minimization problem:
\begin{theorem}
\label{theorem:lq.reform}
Let $\mathcal{G}_{p}^+$ denote the set of $p$ by $p$ lower triangular matrices with positive diagonal elements and unit determinant. Let
\begin{align}
\label{equation:lq.min}
L = \argmin_{\tilde{L} \in \mathcal{G}_{p}^+}||\tilde{L}^{-1}X||,
\end{align}
where $||\cdot||$ is the Frobenius norm. Set $\ell = ||L^{-1}X||$ and $Q = L^{-1}X / \ell$. Then $X = \ell LQ$ is the LQ decomposition of $X$.
\end{theorem}
\begin{proof}
By the uniqueness of the LQ decomposition \citep[Proposition 5.2]{eaton1983multivariate}, it suffices to show that $Q$ has orthonormal rows. We have $QQ^T = I_p \Leftrightarrow L^{-1}XX^TL^{-T}/\ell^2 = I_p \Leftrightarrow XX^T = \ell^2LL^T$. Also note that the solution in (\ref{equation:lq.min}) is equivalent to finding the matrix $\tilde{S}$ that satisfies $\tilde{S} = LL^T = \argmin_{S \in \mathcal{S}_{p}^{1}}\tr(S^{-1}XX^T)$, where $\mathcal{S}_{p}^{1}$ is the set of $p$ by $p$ positive definite matrices with unit determinant. If we can show that $\tilde{S} = XX^T / |XX^T|^{1/p}$ then we have shown that $Q$ has orthonormal rows. Using Lagrange multipliers, we must minimize $\tr(S^{-1}XX^T) - \lambda\log|S|$ in $S \in \mathcal{S}_{p}^+$, the set of $p$ by $p$ positive definite matrices, and $\lambda \in \mathbb{R}$. Equivalently, we could also minimize $\tr(VXX^T) - \lambda\log|V|$, where $V = S^{-1}$. Temporarily ignoring the symmetry of $V$, taking derivatives \citep[chapter 8]{magnus1999matrix} and setting equal to zero we have
\begin{align*}
&XX^T - \lambda V^{-1} = 0 \text{ and } |V| = 1\\
&\Leftrightarrow \lambda V^{-1} = XX^T \text{ and } |V| = 1\\
&\Leftrightarrow V^{-1} = XX^T / |XX^T|^{1/p} \text{ and } \lambda = |XX^T|^{1/p}\\
&\Leftrightarrow S = XX^T / |XX^T|^{1/p} \text{ and } \lambda = |XX^T|^{1/p}.
\end{align*}
Since $\log|V|$ is strictly concave (Theorem 25 of Chapter 11 of \cite{magnus1999matrix} or Theorem 7.6.7 of \cite{horn2012matrix}), $\tr(VXX^T)$ is linear, and $\lambda = |XX^T|^{1/p} > 0$, we have that $\tr(VXX^T) - \lambda\log|V|$ is a convex function in $V$. Hence, $S = XX^T / |XX^T|^{1/p}$ is a global minimum (c.f. Theorem 13 of Chapter 7 in \cite{magnus1999matrix}). Since $XX^T / |XX^T|^{1/p}$ is symmetric and positive definite, it is also a global minimum over the space of symmetric positive definite matrices.
\end{proof}

In (\ref{equation:lq.min}), we are ``dividing out'' $L$ from the rows of $X$. In this way, we can consider the formulation of the LQ decomposition in Theorem \ref{theorem:lq.reform} as finding the $L \in \mathcal{G}_{p}^+$ that accounts for the greatest amount of heterogeneity in the rows of $X$. The goal of accounting for the heterogeneity in each mode of a multidimensional array will lead to our generalization of the LQ decomposition to tensors, where $X \in \mathbb{R}^{p_1\times\cdots\times p_K \times n}$. 
\begin{definition}
If 
\begin{align}
\label{equation:array.opt}
(L_1,\ldots,L_K) = \argmin_{\tilde{L}_k \in \mathcal{G}_{p_k}^+,\ k = 1,\ldots,K}|| (\tilde{L}_1^{-1},\ldots,\tilde{L}_K^{-1},I_n) \cdot X||
\end{align}
then 
\begin{align}
\label{equation:holq}
X = \ell (L_1,\ldots,L_K,I_n) \cdot Q
\end{align}
is an incredible HOLQ, where $\ell = || (L_1^{-1},\ldots,L_K^{-1},I_n) \cdot X||$ and $Q = (L_1^{-1},\ldots,\allowbreak L_K^{-1},I_n) \cdot X / \ell$.
\end{definition}
Here, $(L_1,\ldots,L_K,I_n) \cdot Q$ denotes \emph{multilinear multiplication} of $Q$ by the list of matrices $(L_1,\allowbreak\ldots,L_K,I_n)$ \citep{de2008tensor}, also known as the \emph{Tucker product} \citep{kofidis2001tensor,hoff2011separable}. That is, if $X = (L_1,\ldots,L_K,I_n) \cdot Q$ then
\begin{align*}
X_{[j_1,\ldots,j_K,j_{K+1}]} = \sum_{i_1,\ldots,i_K = 1}^{p_1,\ldots,p_K}Q_{[i_1,\ldots,i_{K},j_{K+1}]}L_{1[j_1,i_1]}\cdots L_{k[j_K,i_K]}.
\end{align*}
Multilinear multiplication has the following useful properties: If (\ref{equation:holq}) holds, then
\begin{align}
&X_{(k)} = L_kQ_{(k)}(I_n \otimes L_K^T \otimes \cdots \otimes L_{k+1}^T \otimes L_{k-1}^T \otimes \cdots \otimes L_1^T) \text{ and} \label{equation:matricization}\\
&\veco(X) = (I_n \otimes L_K^T \otimes \cdots \otimes L_1^T)\veco(Q), \label{equation:vectorization}
\end{align}
where $X_{(k)}$ is the unfolding of the array $X$ into a $p_k$ by $n\prod_{i \neq k}^Kp_i$ matrix and $\veco(X)$ is the unfolding of the array $X$ into a $n\prod_{k=1}^Kp_k$ dimensional vector \citep{kolda2009tensor}. We will generally denote $I_n \otimes L_K \otimes \cdots \otimes L_{k+1} \otimes L_{k-1} \otimes \cdots \otimes L_1$ by $L_{-k}$ and denote $\prod_{k=1}^Kp_k$ by $p$.

We note that such a minimizing $(L_1,\ldots,L_K)$ in (\ref{equation:array.opt}) may not exist. This is discussed further in Section \ref{section:discussion}. When such a minimizer does exist, we may use (\ref{equation:matricization}) and Theorem \ref{theorem:lq.reform} to develop a block coordinate descent algorithm \citep{tseng2001convergence} to solve the minimization problem (\ref{equation:array.opt}): At iteration $i$, we fix $L_k$ for $k \neq i$. We then find the minimizer in $L_i \in \mathcal{G}_{p_i}^+$ of
\begin{align*}
||L_i^{-1}X_{(i)}L_{-i}^{-T}||,
\end{align*}
which, by Theorem \ref{theorem:lq.reform} is the $L$ matrix in the LQ decomposition of $X_{(i)}L_{-i}^{-T} = \ell L Q$. This algorithm is presented in Algorithm \ref{algorithm:flip.flop}. A slight improvement on Algorithm \ref{algorithm:flip.flop} is presented in Algorithm \ref{algorithm:holq} where we also update the core array $Q$ of the  HOLQ while updating the component lower triangular matrices. Unlike Algorithm \ref{algorithm:flip.flop}, Algorithm \ref{algorithm:holq} does not require the calculation of the inverse of $L_k$ or the extra matrix multiplication of $X_{(k)}L_{-k}^{-T}$ at each step. A proof of the equivalence between Algorithms \ref{algorithm:flip.flop} and \ref{algorithm:holq} can be found in the appendix.

\begin{algorithm}[h!]
\caption{Block coordinate descent for the HOLQ.}
\label{algorithm:flip.flop}
  \begin{algorithmic}
    \STATE Given $X \in \mathbb{R}^{p_1 \times \cdots \times p_K \times n}$, initialize:
    \STATE $L_k \leftarrow L_{k0} \in \mathcal{G}_{p_k}^+$ for $k = 1,\ldots,K$.
    \STATE $\ell \leftarrow ||(L_{10}^{-1},\ldots,L_{K0}^{-1},I_n) \cdot X||$
    \REPEAT
    \FOR{$k \in \{1,\ldots,K\}$}
    \STATE LQ decomposition of $X_{(k)}L_{-k}^{-T} = LZ^T$
    \STATE $L_k \leftarrow L / |L|^{1/p_k}$
    \ENDFOR
    \UNTIL{Convergence.}
    \STATE Set $\ell \leftarrow ||(L_1^{-1},\ldots,L_K^{-1},I_n)\cdot X||$
    \STATE Set $Q \leftarrow (L_1^{-1},\ldots,L_K^{-1},I_n)\cdot X / \ell$
    \RETURN $\ell$, $Q$, and $L_k$ for $k = 1,\ldots,K$.
  \end{algorithmic}
\end{algorithm}

\begin{algorithm}[t!]
\caption{Orthogonalized block coordinate descent for the HOLQ.}
\label{algorithm:holq}
  \begin{algorithmic}
    \STATE Given $X \in \mathbb{R}^{p_1 \times \cdots \times p_K \times n}$, initialize:
    \STATE $L_k \leftarrow L_{k0} \in \mathcal{G}_{p_k}^+$ for $k = 1,\ldots,K$.
    \STATE $\ell \leftarrow ||(L_{10}^{-1},\ldots,L_{K0}^{-1},I_n) \cdot X||$
    \STATE $Q \leftarrow (L_{10}^{-1},\ldots,L_{K0}^{-1},I_n) \cdot X/\ell$
    \REPEAT
    \FOR{$k \in \{1,\ldots,K\}$}
    \STATE LQ decomposition of $Q_{(k)} = LZ$
    \STATE $Q_{(k)} \leftarrow Z$
    \STATE $L_k \leftarrow L_kL$
    \STATE Re-scale:
    \begin{description}[noitemsep,nolistsep]
    \item $\ell \leftarrow \ell |L_k|^{1/p_k} ||Q||$
    \item $L_k \leftarrow L_k / |L_k|^{1/p_k}$
    \item $Q \leftarrow Q/||Q||$
    \end{description}
    \ENDFOR
    \UNTIL{Convergence.}
    \RETURN $\ell$, $Q$, and $L_k$ for $k = 1,\ldots,K$.
  \end{algorithmic}
\end{algorithm}

There are two things to note about these algorithms. First, at each iteration we are reducing the criterion function $||(L_1^{-1},\ldots,L_K^{-1},I_n)\cdot X||$. Second, at each iteration of Algorithm \ref{algorithm:holq}, we are orthonormalizing the rows of the core array, $Q$. Hence, the core array $Q$ of any fixed point of this algorithm, including that of the HOLQ, must have a property which we call \emph{scaled all-orthonormality}:

\begin{definition}
A $p_1 \times \cdots \times p_K \times n$ tensor $Q$ is \emph{scaled all-orthonormal} if
\begin{align}
\label{equation:scaled.orthonormal}
Q_{(k)}Q_{(k)}^T = I_{p_k}/p_k \text{ for all } k = 1,\ldots,K.
\end{align}
\end{definition}

\begin{theorem}
Let $X = \ell (L_1,\ldots,L_K,I_n) \cdot Q$ be an incredible HOLQ. Then the core array $Q$ is scaled all-orthonormal.
\end{theorem}
\begin{proof}
This is a direct consequence of the LQ step in Algorithm \ref{algorithm:holq}.
\end{proof}
Note that we divide by $p_k$ in (\ref{equation:scaled.orthonormal}) because of the constraint that $||Q|| = 1$. This scaled all-orthonormality property generalizes the orthonormal rows property in the LQ decomposition.

Of course, we could have instead generalized the RQ decomposition, where for $X \in \mathbb{R}^{p\times n}$ we have $X = RZ$ for $R^T \in \mathcal{G}_{p_k}^+$ and $Z^T \in \mathcal{V}_{p,n}$. For $X \in \mathbb{R}^{p_1\times\cdots\times p_K \times n}$, if $X = \ell(L_1,\ldots,L_K,I_n)\cdot Q$ is the  HOLQ of of $X$, we then take the RQ decomposition of each component $L_k = R_kZ_k$, and set $r = \ell||(Z_1,\ldots,Z_K,I_n)\cdot Q||$ and $Z = \ell(Z_1,\ldots,Z_K,I_n)\cdot Q/r$, then $X = r(R_1,\ldots,R_K,I_n)\cdot Z$ is a higher order RQ (HORQ) of $X$, where $Z$ is scaled all-orthonormal. One could instead have started with a similar minimization formulation of the RQ as we did for the LQ (Theorem \ref{theorem:lq.reform}), then generalize to tensors as we did for the HOLQ (\ref{equation:holq}), and one would obtain the same HORQ as the one we derive from the HOLQ.
\section{The incredible HOLQ for separable covariance inference}
\label{section:mult.norm}
\subsection{Maximum likelihood estimation}
\label{section:MLE}

The LQ decomposition of a data matrix has a close relationship to maximum likelihood inference under the multivariate normal model. Assume a data matrix $X \in \mathbb{R}^{p \times n}$ was generated from a $N_{p \times n}(0,I_n \otimes \Sigma)$ distribution for some $\Sigma$ symmetric and positive definite. That is, the columns of $X$ are assumed to be independently distributed $N_p(0,\Sigma)$ random vectors. The MLE of $\Sigma$ is $XX^T/n$, and so is proportional to $XX^T = LQQ^TL^T = LL^T$, where $X = LQ$ is the LQ decomposition of $X$.

 This result carries over to the multilinear normal model (\ref{equation:mult.norm}) using the  HOLQ. Assume the data array $X \in \mathbb{R}^{p_1\times\cdots\times p_K \times n}$ follows a multilinear normal model, $X \sim \allowbreak N_{p_1\times\cdots\times p_K \times n}(0,\sigma^2I_n\otimes\Sigma_K\otimes\cdots\otimes\Sigma_1)$. That is,
\begin{align}
\label{equation:mult.norm.model}
&X \overset{d}{=} \sigma(\Sigma_1^{1/2},\ldots,\Sigma_K^{1/2},I_n)\cdot Z,
\end{align}
 where $Z \in \mathbb{R}^{p_1\times\cdots\times p_K \times n}$ has i.i.d. standard normal entries and $\Sigma_k^{1/2}$ is the lower triangular Cholesky square root matrix of $\Sigma_k$ for $k = 1,\ldots,K$. Here, we use the identifiable parameterization of \cite{gerard2014equivariant} where $\Sigma_k \in \mathcal{S}_{p_k}^1$ for $k = 1,\ldots,K$ and $\sigma^2 > 0$. The following theorem shows that the MLE of $(\sigma^2,\Sigma_1,\ldots,\Sigma_K)$ can be recovered from the  HOLQ of $X$.

\begin{theorem}
\label{theorem:holq.mle}
Let $X = \ell(L_1,\ldots,L_K,I_n)\cdot Q$ be the incredible HOLQ of $X$. Then under the model (\ref{equation:mult.norm.model})
\begin{enumerate}[noitemsep,nolistsep]
\item The MLE of $\Sigma_k$ is $\hat{\Sigma}_k = L_kL_k^T$ for $k = 1,\ldots,K$,
\item The MLE of $\sigma^2$ is $\hat{\sigma}^2 = \ell^2 / (np)$,
\item The maximized likelihood is equal to 
\begin{align*}
\left(2 \pi \hat{\sigma}^2\right)^{-np/2}e^{-np/2} = \left(2 \pi \ell^2/(np)\right)^{-np/2}e^{-np/2}.
\end{align*}
\end{enumerate}
\end{theorem}
\begin{proof}
The log-likelihood is proportional to 
\begin{align*}
\frac{-np}{2}\log\left(\sigma^2\right) - \frac{1}{2\sigma^2} ||(\Sigma_1^{-1/2},\ldots,\Sigma_K^{-1/2},I_n)\cdot X||^2,
\end{align*}
where $\Sigma_k^{1/2}$ is the lower triangular Cholesky square root matrix of $\Sigma_k$. Holding the $\Sigma_k$'s fixed, taking a derivative of $\sigma^2$ and setting equal to zero, we solve for $\sigma^2$ and obtain $\hat{\sigma}^2 = ||(\Sigma_1^{-1/2},\ldots,\Sigma_K^{-1/2},I_n)\cdot X||^2 / (np)$. A second derivative test confirms this is the global maximizer for any fixed $\Sigma_1,\ldots,\Sigma_K$. The profiled likelihood is then
\begin{align}
\label{equation:maximized.like}
&\left(2\pi\hat{\sigma}^2\right)^{-np/2}\exp\left\{-\frac{1}{2\hat{\sigma}^2} ||(\Sigma_1^{-1/2},\ldots,\Sigma_K^{-1/2},I_n)\cdot X||^2\right\} \nonumber\\
&=\left(2\pi\hat{\sigma}^2\right)^{-np/2}\exp\left\{-\frac{1}{2\hat{\sigma}^2} \hat{\sigma}^2 np\right\}\nonumber\\
&=\left(2 \pi \hat{\sigma}^2\right)^{-np/2}e^{-np/2}.
\end{align}
Thus, to maximize the likelihood, we must minimize $\hat{\sigma}^2 = \frac{1}{np}|| (\Sigma_1^{-1/2},\ldots,  \allowbreak \Sigma_K^{-1/2}, \allowbreak I_n)\cdot X||^2$ in $\Sigma_k^{1/2} \in \mathcal{G}_{p_k}^+$ for $k = 1,\ldots,K$. This is the same as the minimization problem solved by the HOLQ in (\ref{equation:array.opt}). Hence, the MLE of $\Sigma_k$ is $\hat{\Sigma}_k = L_kL_k^T$. This in turn implies that $\hat{\sigma}^2 = ||(\hat{\Sigma}_1^{-1/2},\ldots,\hat{\Sigma}_K^{-1/2},I_n)\cdot X||^2 / (np)  = ||(L_1^{-1},\ldots,L_K^{-1},I_n)\cdot X||^2 / (np) = \ell^2/(np)$. We may plug $\hat{\sigma}^2 = \ell^2/(np)$ into (\ref{equation:maximized.like}) to obtain the final part of the theorem.
\end{proof}

This relationship with the multilinear normal model extends to any array-variate elliptically contoured model with separable covariance. Using our identifiable parameterization, $X$ is a mean zero elliptically contoured random array with separable covariance if its density has the form
\begin{align*}
f(x|\sigma^2,\Sigma_1,\ldots,\Sigma_K) \propto (\sigma^2)^{-p/2}g(||(\Sigma_1^{-1/2},\ldots,\Sigma_K^{-1/2})\cdot x||^2/\sigma^2),
\end{align*}
for some known $g:\mathbb{R}^+ \rightarrow \mathbb{R}^+$. Using a general result of \cite{anderson1986maximum} (see \ref{section:appendix.anderson}), the MLE of $\sigma^2(\Sigma_K \otimes \cdots \otimes \Sigma_1)$ can be shown to be proportional to the MLE under the multilinear normal model. This in turn implies that the MLEs of the component covariance matrices in separable elliptically contoured distributions have the same relationship with the  HOLQ as in the multilinear normal model. That is, $\hat{\Sigma}_k = L_kL_k^T$ where $X = \ell(L_1,\ldots,L_K,I_n)\cdot Q$. Only the estimation of the scale $\sigma^2$ might be different, depending on the function $g$.


The MLEs of $\sigma^2$ and the $\Sigma_k$'s depend only on $\ell$ and the $L_k$'s, not $Q$. This suggests that the core array $Q$ might be ancillary with respect to the covariance parameters $\Sigma_1,\ldots,\Sigma_K$ and $\sigma^2$, that is, the distribution of $Q$ might not depend on the parameter values. In the next paragraph, we will prove that this is indeed the case, but to do so we first introduce a group of transformations that acts transitively on the parameter space. Consider the group 
\begin{align*}
\mathcal{G} = \{(a,A_1,\ldots,A_K) : a > 0, A_k \in \mathcal{G}_{p_k}^+ \text{ for } k=1,\ldots,K\},
\end{align*}
where the group operation is component-wise multiplication. For example, if $(a,A_1,\ldots,A_K)$, $(b,B_1,\ldots,B_K)\in \mathcal{G}$, then we have
\begin{align*}
(a,A_1,\ldots,A_K)(b,B_1,\ldots,B_K) = (ab,A_1B_1,\ldots,A_KB_K).
\end{align*}
The group acts on the sample space by
\begin{align*}
X \mapsto a(A_1,\ldots,A_K,I_n)\cdot X.
\end{align*}
The following theorem shows that under this group action, the core array of the HOLQ, if unique, is maximally invariant (uniqueness is discussed briefly in Section \ref{section:discussion}). More generally, this theorem states that the set of core arrays of fixed points from Algorithm \ref{algorithm:holq} is a maximally invariant statistic. In other words, two arrays are in the same orbit of $\mathcal{G}$ if and only if the set of core arrays of fixed points of Algorithm \ref{algorithm:holq} are the same.
\begin{theorem}
\label{theorem:max.inv}
Let $X$ and $Y$ be in $\mathbb{R}^{p_1\times\cdots p_K\times n}$. Let $\mathcal{Q}_X$  and $\mathcal{Q}_Y$ be the set of core arrays from fixed points of Algorithm \ref{algorithm:holq} for $X$ and $Y$, respectively. Then $\mathcal{Q}_X = \mathcal{Q}_Y$ if and only if there exist  $c > 0$ and $C_k \in \mathcal{G}_{p_k}^+$ for $k = 1,\ldots,K$ such that $c(C_1,\ldots,C_K,I_n)\cdot X = Y$. 
\end{theorem}
\begin{proof}
We first prove the ``only if'' part. Assume that $\mathcal{Q}_X = \mathcal{Q}_Y$, then we choose one $Q$ in $\mathcal{Q}_X = \mathcal{Q}_Y$. Then there exists $a,b > 0$ and $A_k,B_k \in \mathcal{G}_{p_k}^+$ for $k = 1,\ldots,K$ such that $X = a(A_1,\ldots,A_K,I_n)\cdot Q$ and $Y = b(B_1,\ldots,B_K,I_n)\cdot Q$. One may set $c = b/a$ and $C_k = B_kA_k^{-1}$ to prove that $c(C_1,\ldots,C_K,I_n) \cdot X = Y$.

We now prove the ``if'' part. Assume there exist $c > 0$ and $C_k \in \mathcal{G}_{p_k}^+$ for $k = 1,\ldots,K$ such that $c(C_1,\ldots,C_K,I_n)\cdot X = Y$. Then for each $Q$ in $\mathcal{Q}_X$ we have that $Y = ca(C_1A_1,\ldots,C_KA_K,I_n)\cdot Q$ for some $a > 0$ and $A_k \in \mathcal{G}_{p_k}^+$ for $k = 1,\ldots,K$. Since fixed points are entirely determined by the scaled all-orthonormality of the core, $Q$ is also in $\mathcal{Q}_Y$. Likewise any $Q$ in $\mathcal{Q}_Y$ will also be in $\mathcal{Q}_X$. Hence $\mathcal{Q}_X = \mathcal{Q}_Y$.
\end{proof}

By using the above invariance results, we may now prove that $\mathcal{Q}_{X}$ is ancillary. The group $\mathcal{G}$ acts on the parameter space by \citep{hoff2011separable}
\begin{align*}
\sigma^2 \mapsto a^2\sigma^2 \text{ and } \Sigma_k \mapsto A_k\Sigma_kA_k^T.
\end{align*}
This action is clearly transitive over the parameter space. Hence, the maximally invariant parameter is a constant. Since the distribution of any invariant statistic depends only on the maximally invariant parameter \citep[Theorem 6.3.2]{lehmann2006testing}, the distribution of $\mathcal{Q}_X$ is ancillary with respect to $\sigma^2$ and $\Sigma_k$ for $k = 1,\ldots,K$. If the MLE is unique, then the core array of the HOLQ is in 1-1 correspondence with $\mathcal{Q}_X$, and so is also maximally invariant. Hence, the core array from a unique HOLQ is ancillary with respect to the covariance parameters, $\Sigma_1,\ldots,\Sigma_K$, and $\sigma^2$. This result holds not just for elliptically contoured array-variate models with separable covariance, but also for models of the form
\begin{align}
\label{equation:sep.cov}
X \overset{d}{=} \sigma (\Sigma_1^{1/2},\ldots,\Sigma_K^{1/2},I_n)\cdot Z,
\end{align}
where $Z$ has a fixed distribution such that $E[Z] = 0$, $\cov(\veco(Z)) = I_{np}$, and $\Sigma_k^{1/2}$ is the lower triangular Cholesky square root of $\Sigma_k$. 


\subsection{HOLQ juniors}
\label{section:holq.junior}
If it is believed that the dependencies along a mode follow a particular pattern, then from the perspective of parameter estimation, it would make sense to fit a structured covariance matrix that corresponds to the pattern along that mode.  For example, if it is believed that the ``slices'' of the array along a particular mode $k$ are statistically independent, then one would use a model with $\Sigma_k$ restricted to be a diagonal matrix. If the $p_k$ slices along the mode $k$ are believed to be i.i.d., then one could restrict $\Sigma_k$ to be the identity matrix. If one of the modes $k$ corresponded to data gathered over sequential time points, then one could fit $\Sigma_k$ to correspond to an auto-regressive covariance model, such as that of containing constant prediction error variances and arbitrary autoregressive coefficients. One could then restrict $\Sigma_k$ to have its lower triangular Cholesky square root to have unit diagonal \citep{pourahmadi1999joint}. Each of these alternatives corresponds to fitting a submodel of an unrestricted separable covariance model.

We represent such submodels mathematically as follows: Partition the index set $\left\{1,\ldots,K\right\}$ into four non-overlapping sets $J_1, J_2, J_3, J_4$. Let $\mathcal{D}_{p_k}^+$ denote the group of $p_k$ by $p_k$ positive definite diagonal matrices with unit determinant. Also, let $\mathcal{S}_{p_k}^{Ch}$ be the space of $p_k$ by $p_k$ symmetric and positive definite matrices whose lower triangular Cholesky square roots have unit diagonal.  Assume the model $X \sim N_{p_1\times\cdots\times p_K}(0,\sigma^2\Sigma_K\otimes\cdots\otimes\Sigma_1)$ where $\Sigma_k$ is in $\mathcal{S}_{p_k}^1$, $\mathcal{D}_{p_k}^+$, $\mathcal{S}_{p_k}^{Ch}$, or $\{I_{p_k}\}$ when $k$ is in $J_1$, $J_2$, $J_3$, or $J_4$, respectively. The collection of sets $J_1, J_2, J_3,$ and $J_4$ corresponds to a submodel where the modes in $J_1$ have unrestricted covariance, the modes in $J_2$ have diagonal covariance,  the modes in $J_3$  have constant prediction error variances and arbitrary autoregressive coefficients, and the modes in $J_4$ have independence and homoscedastic covariance structure. If such a submodel represents a close approximation to the truth, then one would expect to obtain better estimates by fitting this submodel than by fitting an unrestricted multilinear normal model.


In the same way that the HOLQ provides the MLEs in the multilinear normal model, the MLEs in submodels of the unconstrained multilinear normal model are provided by a class of Tucker decompositions we call HOLQ juniors. A HOLQ junior is found by constraining the component matrices in the Tucker decomposition to be in a subspace of $\mathcal{G}_{p_k}^+$. In particular, we consider constraining each $L_k$ in (\ref{equation:holq}) to be in $\mathcal{G}_{p_k}^+$, $\mathcal{D}_{p_k}^+$, $\mathcal{G}_{p_k}^{Ch}$, or $\{I_{p_k}\}$, where $\mathcal{G}_{p_k}^{Ch}$ denotes the set of $p_k$ by $p_k$ lower triangular matrices with unit diagonal.

\begin{definition}[HOLQ junior]
\label{theorem:holq.junior}
Let $\mathcal{G}^{(k)} = \mathcal{G}_{p_k}^+$, $\mathcal{D}_{p_k}^+$, $\mathcal{G}_{p_k}^{Ch}$, or $\{I_{p_k}\}$ if $k$ is in $J_1$, $J_2$, $J_3$, or $J_4$, respectively. If
\begin{align*}
(L_1,\ldots,L_K) = \argmin_{\tilde{L}_k \in \mathcal{G}^{(k)},\ k=1,\ldots,K}|| (\tilde{L}_1^{-1},\ldots,\tilde{L}_K^{-1}) \cdot X||,
\end{align*}
then 
\begin{align}
\label{equation:holq.junior}
X = \ell (L_1,\ldots,L_K) \cdot Q
\end{align}
is a HOLQ junior, where $\ell = || (L_1^{-1},\ldots,L_K^{-1}) \cdot X||$ and $Q = (L_1^{-1},\ldots,L_K^{-1}) \cdot X / \ell$.
\end{definition}
The core array of a HOLQ junior also has a special structure that we prove in the following theorem.
\begin{theorem}
\label{theorem:holq.junior.core}
Let $X = \ell (L_1,\ldots,L_K) \cdot Q$ be a HOLQ junior (\ref{equation:holq.junior}). Then the core array has the following properties:
\begin{enumerate}[noitemsep,nolistsep]
\item $Q_{(k)}Q_{(k)}^T = I_{p_k}/p_k$ for all $k \in J_1$,
\item $\diag\left(Q_{(k)}Q_{(k)}^T\right) = \mathbf{1}_{p_k}/p_k$ for all $k \in J_2$, where $\mathbf{1}_{p_k} \in \mathbb{R}^{p_k}$ is the vector of $1$'s, and
\item $Q_{(k)}Q_{(k)}^T = D_k$ for some diagonal matrix $D_k$ for all $k \in J_3$.
\end{enumerate}
\end{theorem}
\begin{proof}
 We may update the modes for which $k \in J_1$ using Theorem \ref{theorem:lq.reform} the same way we did in Algorithm \ref{algorithm:holq}. The core array of any fixed point must then have the property that $Q_{(k)}Q_{(k)}^T = I_{p_k}/p_k$ for all $k \in J_1$. The proofs for $k \in J_2$ and $k \in J_3$ follow along the same lines as in the proof for $k \in J_1$, and are in the appendix.
\end{proof}

The same arguments as used in Section \ref{section:MLE} show that maximum likelihood inference in multilinear normal submodels has a close connection with HOLQ juniors. The proof of the following is very similar to that of Theorem \ref{theorem:holq.mle} and is omitted. 
\begin{theorem}
\label{theorem:mle.holq.junior}
Let $X = \ell(L_1,\ldots,L_K)\cdot Q$ be a HOLQ junior. We assume the model $X \sim N_{p_1\times\cdots\times p_K}(0,\allowbreak\sigma^2\Sigma_K\otimes\cdots\otimes\Sigma_1)$ where $\Sigma_k$ is in $\mathcal{S}_{p_k}^1$, $\mathcal{D}_{p_k}$, $\mathcal{S}_{p_k}^{Ch}$, or $\{I_{p_k}\}$ when $k$ is in $J_1$, $J_2$, $J_3$, or $J_4$, respectively. We have the following:
\begin{enumerate}[noitemsep,nolistsep]
\item The MLE of $\Sigma_k$ is $L_kL_k^T$ for $k = 1,\ldots,K$,
\item The MLE of $\sigma^2$ is $\ell^2 / (np)$,
\item The maximum of the likelihood is equal to 
\begin{align*}
\left(2 \pi \hat{\sigma}^2\right)^{-np/2}e^{-np/2} = \left(2 \pi \ell^2/(np)\right)^{-np/2}e^{-np/2}.
\end{align*}
\end{enumerate}
\end{theorem}

We note here that the same group invariance arguments as used in Section \ref{section:MLE} prove that the core array from a unique HOLQ junior is ancillary with respect to the covariance parameters in separable covariance models. That is, a core array from a unique HOLQ junior (\ref{equation:holq.junior}) is ancillary under the model
\begin{align}
\label{equation:sep.cov}
X \overset{d}{=} \sigma (\Sigma_1^{1/2},\ldots,\Sigma_K^{1/2})\cdot Z,
\end{align}
where $Z$ has a fixed distribution such that $E[Z] = 0$, $\cov(\veco(Z)) = I_{p}$, and $\Sigma_k^{1/2}$ is the lower Cholesky square root of $\Sigma_k$ in $\mathcal{S}_{p_k}^1$, $\mathcal{D}_{p_k}^+$, $\mathcal{S}_{p_k}^{Ch}$, or $\{I_{p_k}\}$ when $k$ is in $J_1$, $J_2$, $J_3$, or $J_4$, respectively. Equivalently, $\Sigma_k^{1/2}$ is in $\mathcal{G}_{p_k}^+$, $\mathcal{D}_{p_k}^+$, $\mathcal{G}_{p_k}^{Ch}$, or $\{I_{p_k}\}$ when $k$ is in $J_1$, $J_2$, $J_3$, or $J_4$, respectively

\subsection{Likelihood ratio testing}
\label{section:LRT}
One would expect to lose efficiency in covariance estimation when fitting a large model when a submodel is a close approximation to the truth. To aid modeling decisions, we develop a class of likelihood ratio tests (LRTs) for comparing nested separable models. For example, a test of independence across slices of mode $k$ would correspond to $H_0: \Sigma_k \in \mathcal{D}_{p_k}^+$ versus $H_1: \Sigma_k \in \mathcal{S}_{p_k}^1$. A test for independence and heteroscedasticity against independence and homoscedasticity along mode $k$ would correspond to $H_0: \Sigma_k = I_{p_k}$ versus $H_1: \Sigma_k \in \mathcal{D}_{p_k}^+$. In a longitudinal setting, testing for the presence of non-zero autoregressive coefficients along mode $k$ would correspond to $H_0: \Sigma_k = I_{p_k}$ versus $H_1: \Sigma_k \in \mathcal{S}_{p_k}^{Ch}$.  As seen in Section \ref{section:holq.junior}, each submodel of the unstructured multilinear normal model corresponds to a HOLQ junior. If we have two models $H_0$ and $H_1$, with $H_0$ nested in $H_1$, then the likelihood ratio test takes on the simple form of the ratio of the two scale estimates of the HOLQ juniors corresponding to $H_0$ and $H_1$.

\begin{theorem}
\label{theorem:lrt.stat}
Suppose $H_0$ is a submodel of $H_1$. Suppose $\veco(X) = \ell(L_K\otimes\cdots\otimes L_1)\veco(Q)$ and $\veco(X) = a(A_M\otimes\cdots\otimes A_1)\veco(Z)$ are two HOLQ juniors in vectorized form (\ref{equation:vectorization}) corresponding to $H_0$ and $H_1$, respectively. Hence, $\hat{\sigma}_0^2 = \ell^2/p$ and $\hat{\sigma}_1^2 = a^2/p$ are the MLEs of the scale parameters under $H_0$ and $H_1$, respectively. Then the LRT of $H_0$ versus $H_1$ rejects for large values of $\hat{\sigma}_0^2/\hat{\sigma}_1^2$, or equivalently $\ell / a$.
\end{theorem}
\begin{proof}
Applying Theorem 1 from \cite{anderson1986maximum}  and Theorem \ref{theorem:mle.holq.junior} (see \ref{section:appendix.anderson}), the LRT rejects for large values of 
\begin{align*}
\hat{\sigma}_1^{-p} / \hat{\sigma}_0^{-p} = a^{-p}/\ell^{-p} = \ell^p/a^p,
\end{align*}
or, equivalently, for large values of $\ell / a$.
\end{proof}
The LRT in Theorem \ref{theorem:lrt.stat} includes testing for a Kronecker structured covariance matrix along modes $k$ and $j$ against an unrestricted covariance matrix along the concatenated modes of $k$ and $j$. That is, it allows for the test $H_0: \Sigma_{kj} = \Sigma_k \otimes \Sigma_j$ for $\Sigma_k \in \mathcal{S}_{p_k}^1$ and $\Sigma_j \in \mathcal{S}_{p_j}^1$ versus $H_1: \Sigma_{ij} \in \mathcal{S}_{p_kp_j}^1$. This is why $M$ may be different from $K$. For example, if all modes in $H_0$ and $H_1$ had the same covariance structure except modes $k$ and $j$, for which $H_0$ assumes has separable covariance and for which $H_1$ assumes has unstructured covariance along the concatenated mode $kj$, then $M = K-1$. This particular type of test is useful for determining how much separability is reasonable to assume in a covariance matrix.

The likelihood ratio test has a nice intuitive interpretation. Since the MLE of $\sigma^2$ under $H_0$ is $\hat{\sigma}_0^2 = \ell^2/p = ||(L_1^{-1},\ldots,L_K^{-1})\cdot X||^2/p$ (Theorem \ref{theorem:mle.holq.junior}), one can consider $\hat{\sigma}_0^2$ as a sort of mean squares left after accounting for covariance/heterogeneity along modes $1,\ldots,K$. Likewise $\hat{\sigma}_1^2$ is a sort of mean squares left after accounting for covariance/heterogeneity along modes $1,\ldots,M$. The likelihood ratio test rejects the null when we can explain significantly more heterogeneity in $X$ by increasing the complexity of the covariance structure.

For many hypothesis tests, the distribution of $p\left(\log\left(\ell^2\right) - \log\left(a^2\right)\right)$, the log-likelihood ratio statistic, can be approximated by a $\chi^2$ distribution. However, this asymptotic approximation would be suspect for small sample sizes. We propose using a Monte Carlo approximation to the null distribution of the LRT statistic. This Monte Carlo approximation can be made arbitrarily precise. The following theorem, whose proof is in the appendix, suggests how to sample from the null distribution of the LRT statistic, $\ell / a$, or $\hat{\sigma}_0/\hat{\sigma}_1$, in Theorem \ref{theorem:lrt.stat}.
\begin{theorem}
\label{theorem:lrt.dist}
Under $H_0$, the distribution of $\ell/a$ in Theorem \ref{theorem:lrt.stat} does not depend on the parameter values $\Sigma_1,\ldots,\Sigma_K,$ and $\sigma^2$.
\end{theorem}
This property of the LRT statistic was noted by \cite{lu2005likelihood} for the matrix-normal case. An immediate implication of Theorem \ref{theorem:lrt.dist} is that for tests of these covariance models, a Monte Carlo sample of the LRT statistic under $H_0$ can be made by simulating values of $\ell/a$ under $H_0$. A single value of $\ell/a$ may be simulated from $H_0$ as follows:
\begin{enumerate}[noitemsep,nolistsep]
\item sample $x \sim N_{p}\left(0,I_{p}\right)$,
\item construct $X_1 \in \mathbb{R}^{p_1 \times \cdots \times p_K}$ and $X_2 \in \mathbb{R}^{q_1\times \cdots \times q_M}$ from $x$,
\item calculate HOLQ juniors $X_1 = \ell(L_1,\ldots,L_K)\cdot Q$ and $X_2 = a(A_1,\ldots,\allowbreak A_M)\cdot Z$,
\item calculate $\ell/a$.
\end{enumerate}

\section{Other tensor decompositions}
\label{section:other.tensor}
\subsection{The incredible SVD}
\label{section:isvd}
The incredible HOLQ (\ref{equation:holq}) may be used to derive a higher order analogue to the SVD that is related to the HOSVD of \cite{de2000multilinear,de2000best}. From (\ref{equation:holq}), we take the SVD of each component lower triangular matrix, $L_k = U_kD_kV_k^T$ for $k = 1,\ldots,K$. Letting $V = (V_1^T,\ldots,V_K^T,I_n) \cdot Q$, we now have an exact decomposition of the data array $X$ which may be viewed as a higher order generalization of the SVD.

\begin{definition}
\label{def:isvd}
Suppose
\begin{align}
\label{equation:ISVD}
X = \ell (U_1,\ldots,U_K,I_n) \cdot [(D_1,\ldots,D_K,I_n) \cdot V]
\end{align}
such that
\begin{enumerate}[noitemsep,nolistsep]
\item $\ell \geq 0$,
\item $U_k \in \mathcal{O}_{p_k}$, the set of $p_k$ by $p_k$ orthogonal matrices, for all $k = 1,\ldots,K$,
\item $D_k \in \mathcal{D}_{p_k}^+$, for all $k = 1,\ldots,K$, and
\item $V$ is scaled all-orthonormal.
\end{enumerate}
Then we say that (\ref{equation:ISVD}) is an incredible SVD (ISVD).
\end{definition}

The ISVD can be seen as a type of ``core rotation'' \citep{kolda2009tensor} of the HOSVD. The core is rotated to a form where we may separate the ``mode specific singular values'', $D_1,\ldots,D_K$, from the core. Where the core array in the HOSVD is all-orthogonal (the mode-$k$ unfolding contains orthogonal, but not necessarily orthonormal, rows for all $k = 1,\ldots,K$), the core array in the ISVD is scaled all-orthonormal.

A low rank version of the ISVD can be defined by finding, for $r_k \leq p_k$ for $k = 1,\ldots,K$, the $U_k \in \mathcal{V}_{r_k,p_k}$, $D_k \in \mathcal{D}_{r_k}^+$ for $k = 1,\ldots,K$, $\ell > 0$, and $V \in \mathbb{R}^{r_1\times\cdots\times r_K \times n}$ that minimize
\begin{align}
\label{equation:fnorm}
||X - \ell (U_1,\ldots,U_K,I_n) \cdot [(D_1,\ldots,D_K,I_n) \cdot V]||^2.
\end{align}
We can apply the HOOI \citep[higher-order orthogonal iteration,][]{de2000best} to obtain the minimizer of (\ref{equation:fnorm}). Let $X = (V_1,\ldots,V_K,I_n) \cdot S$ be the HOOI of $X$. This minimizes
\begin{align*}
||X - (V_1,\ldots,V_K,I_n) \cdot S||^2,
\end{align*}
for arbitrary core array $S \in \mathbb{R}^{r_1\times\cdots\times r_K,n}$ and arbitrary $V_k \in \mathcal{V}_{r_k,p_k}$. We now take the ISVD of $S = \ell (W_1,\ldots,W_K,I_n) \cdot [(D_1,\ldots,D_K,I_n) \cdot V]$. We set $U_k = V_kW_k$ for $k = 1,\ldots,K$. These values now minimize (\ref{equation:fnorm}). The truncated ISVD does not improve the fit of the low rank array to the data array over the HOOI. Rather, the truncated ISVD can be seen as a core rotation of the HOOI, the same as how the ISVD can be seen as a core rotation of the HOSVD. Again, the core is rotated to a form where we may separate the mode specific singular values, $D_1,\ldots,D_K$, from the core.



\subsection{The IHOP decomposition}
\label{section:ihop}
In this section, we explore how our minimization approach may lead to another novel Tucker decomposition. Let $X$ be a $p$ by $n$ matrix with $p \leq n$ such that $X$ is of rank $p$. We may write $X$ as 
\begin{align*}
X = PW,
\end{align*}
where $P \in \mathcal{S}_{p}^+$ and $W^T \in \mathcal{V}_{p,n}$. This is known as the (left) \emph{polar decomposition} (see, for example, Proposition 5.5 of \cite{eaton1983multivariate}). Following the theme of this paper, we reformulate the polar decomposition as a minimization problem. Let $\mathcal{S}_{p}^F$ denote the space of $p$ by $p$ positive definite matrices with unit trace.
\begin{theorem}
\label{theorem:polar.min}
Let
\begin{align}
\label{equation:polar.min}
P = \argmin_{\tilde{P} \in \mathcal{S}_{p}^F} \tr(\tilde{P}^{-1}XX^T).
\end{align}
Set $\ell = ||P^{-1}X||$ and $W = P^{-1}X/\ell$. Then 
\begin{align*}
X = \ell P W
\end{align*}
is the polar decomposition of $X$.
\end{theorem}
\begin{proof}
By the uniqueness of the polar decomposition \citep[Proposition 5.5]{eaton1983multivariate}, it suffices to show that $W$ has orthonormal rows. We have that $WW^T = I_{p} \Leftrightarrow P^{-1}XX^TP^{-1}/\ell^2 = I_p \Leftrightarrow XX^T = \ell^2PP$. Hence, if we can show that $PP \propto XX^T$ then we have shown that $W$ has orthonormal rows. Using Lagrange multipliers, we must minimize $\tr(P^{-1}XX^T) + \lambda(\tr(P) - 1)$. This is equivalent to minimizing $\tr(VXX^T) + \lambda(\tr(V^{-1}) - 1)$ where $V = P^{-1}$. Temporarily ignoring the symmetry, taking derivatives, and setting equal to $0$, we have
\begin{align*}
&XX^T - \lambda V^{-1}V^{-1}  = 0 \text{ and } \tr(V^{-1}) = 1\\
&\Leftrightarrow XX^T = \lambda V^{-1}V^{-1} \text{ and } \tr(V^{-1}) = 1\\
&\Rightarrow V^{-1} = (XX^T)^{1/2} / \tr((XX^T)^{1/2}) \text{ and } \lambda = \tr((XX^T)^{1/2})^2\\
&\Rightarrow P = (XX^T)^{1/2} / \tr((XX^T)^{1/2}) \text{ and } \lambda = \tr((XX^T)^{1/2})^2,
\end{align*}
 where $(XX^T)^{1/2}$ is any square root matrix of $XX^T$. Let $(XX^T)^{1/2}$ now be the unique symmetric square root matrix of $XX^T$, which is a critical point of $\tr(VXX^T) + \lambda(\tr(V^{-1}) - 1)$ over the space of positive definite matrices. From problem 2 of Section 7.6 in \cite{horn2012matrix}, we have that $\tr(V^{-1})$ is strictly convex on the set of positive definite matrices. Since $\lambda = \tr((XX^T)^{1/2})^2 > 0$, we have that $\tr(VXX^T) + \lambda(\tr(V^{-1}) - 1)$ is a convex function for all positive definite $V$. Therefore $P = (XX^T)^{1/2} / \tr((XX^T)^{1/2})$ is a global minimum (c.f. Theorem 13 of Chapter 7 in \cite{magnus1999matrix}).
\end{proof}
For $X \in \mathbb{R}^{p_1\times\cdots\times p_K \times n}$, we now define the incredible higher order polar decomposition (IHOP).
\begin{definition}
If
\begin{align}
\label{equation:ihop.min}
(P_1,\ldots,P_K) = \argmin_{P_k\in\mathcal{S}_{p_k}^F, k = 1,\ldots,K} \tr[(P_K^{-1}\otimes\cdots\otimes P_1^{-1})X_{(K+1)}^TX_{(K+1)}],
\end{align}
then 
\begin{align*}
X = \ell(P_1,\ldots,P_K,I_n)\cdot W
\end{align*}
is an IHOP, where $\ell = ||(P_1^{-1},\ldots,P_K^{-1},I_n)\cdot X||$ and $W = (P_1^{-1},\ldots,P_K^{-1},I_n)\cdot X / \ell$. 
\end{definition}

Let $\mathcal{G}_{p}^F$ be the space of lower triangular matrices with positive diagonal elements and unit Frobenius norm. To derive a block coordinate descent algorithm to find the solution to (\ref{equation:ihop.min}), we note that (\ref{equation:polar.min}) is equivalent to finding the $L \in \mathcal{G}_{p}^F$ such that
\begin{align*}
L = \argmin_{\tilde{L}\in\mathcal{G}_{p}^F}||\tilde{L}^{-1}X||,
\end{align*}
and then setting $P = LL^T$ for $P$ from (\ref{equation:polar.min}). Hence, (\ref{equation:ihop.min}) is equivalent to finding $L_k \in \mathcal{G}_{p_k}^F$ for $k = 1,\ldots,K$ such that
\begin{align}
(L_1,\ldots,L_K) = \argmin_{\tilde{L}_k\in\mathcal{G}_{p_k}^F, k = 1,\ldots,K}||(\tilde{L}_1^{-1},\ldots,\tilde{L}_K^{-1},I_n)\cdot X||,
\end{align}
then setting $P_k = L_kL_k^T$ for $k = 1,\ldots,K$. At iteration $i$, fix $L_k$ for $k \neq i$. We then find the minimizer in $L_i \in \mathcal{G}_{p_i}^F$ of
\begin{align*}
||L_i^{-1}X_{(i)}L_{-i}^{-T}|| = \tr(P_i^{-1}X_{(i)}P_{-i}^{-1}X_{(i)}^T),
\end{align*}
which, by Theorem \ref{theorem:polar.min} is $L \in \mathcal{G}_{p_k}^F$ such that $LL^TW = X_{(i)}L_{-i}^{-1}$ is the polar decomposition of $X_{(i)}L_{-i}^{-1}$. This algorithm is presented in Algorithm \ref{algorithm:ihop}. Again following the theme in this paper, we present a slightly improved algorithm in Algorithm \ref{algorithm:pancake}. A proof that Algorithm \ref{algorithm:ihop} and Algorithm \ref{algorithm:pancake} are equivalent can be found in the appendix. From the Algorithm \ref{algorithm:pancake}, we see that any fixed point of $R$ in Algorithm \ref{algorithm:pancake} must have the property that $R_{(k)} = L_kZ$ for the current value of $L_k$ and some $Z$ with orthonormal rows.

\begin{algorithm}[h!]
\caption{Block coordinate descent for the IHOP.}
\label{algorithm:ihop}
  \begin{algorithmic}
    \STATE Given $X \in \mathbb{R}^{p_1 \times \cdots \times p_K \times n}$, initialize:
    \STATE $L_k \leftarrow L_{k0} \in \mathcal{G}_{p_k}^F$ for $k = 1,\ldots,K$.
    \REPEAT
    \FOR{$k \in \{1,\ldots,K\}$}
    \STATE Polar decomposition of $X_{(k)}L_{-k}^{-1} = PZ^T$
    \STATE Cholesky decomposition of $P = LL^T$
    \STATE $L_k \leftarrow L / ||L||$
    \ENDFOR
    \UNTIL{Convergence.}
    \STATE Set $P_k \leftarrow L_kL_k^T$ for $k = 1,\ldots,K$
    \STATE Set $\ell \leftarrow ||(P_1^{-1},\ldots,P_K^{-1},I_n)\cdot X||$
    \STATE Set $W \leftarrow (P_1^{-1},\ldots,P_K^{-1},I_n)\cdot X / \ell$
    \RETURN $\ell$, $W$, and $P_k$ for $k = 1,\ldots,K$.
  \end{algorithmic}
\end{algorithm}

\begin{algorithm}[ht!]
\caption{Orthogonalized block coordinate descent for the IHOP.}
\label{algorithm:pancake}
\begin{algorithmic}
    \STATE Given $X \in \mathbb{R}^{p_1 \times \cdots \times p_K \times n}$, initialize:
    \STATE $L_k \leftarrow L_{k0} \in \mathcal{G}_{p_k}^F$ for $k = 1,\ldots,K$.
    \STATE $\ell \leftarrow ||(L_1^{-1},\ldots,L_K^{-1},I_n)\cdot X||$
    \STATE $R \leftarrow (L_1^{-1},\ldots,L_K^{-1},I_n)\cdot X / \ell$
    \REPEAT
    \FOR{$k \in \{1,\ldots,K\}$}
    \STATE Polar decomposition of $L_kR_{(k)} = PZ$
    \STATE Cholesky decomposition of $P = LL^T$
    \STATE Set $R_{(k)} \leftarrow L^TZ$
    \STATE Set $L_k \leftarrow L$
    \STATE Re-scale:
    \begin{description}[noitemsep,nolistsep]
    \item $\ell \leftarrow \ell ||L_k|| ||R||$
    \item $L_k \leftarrow L_k / ||L_k||$
    \item $R \leftarrow R/||R||$
    \end{description}
    \ENDFOR
    \UNTIL{Convergence.}
    \STATE Set $P_k \leftarrow L_kL_k^T$ for $k = 1,\ldots,K$
    \STATE Set $\ell \leftarrow ||(L_1^{-1},\ldots,L_K^{-1})\cdot R||$
    \STATE Set $W = (L_1^{-1},\ldots,L_K^{-1})\cdot R / \ell$
    \RETURN $\ell$, $W$, and $P_k$ for $k = 1,\ldots,K$.
\end{algorithmic}
\end{algorithm}

\section{Discussion}
\label{section:discussion}
In this paper, we have presented a higher order generalization of the LQ decomposition by reformulating the LQ decomposition as a minimization problem. The orthonormal rows property of the $Q$ matrix in the LQ decomposition generalizes to the scaled all-orthonormal property of the mode-$k$ unfoldings of the core array in the  HOLQ. We generalized the HOLQ to HOLQ juniors by constraining the component matrices to subspaces of $\mathcal{G}_{p_k}^+$. One application of the  HOLQ (junior) is for estimation and testing in separable covariance models. The MLEs of the covariance parameters may be recovered from the HOLQ (junior) and the likelihood ratio test has the simple form of the ratio of two scale estimates from the HOLQ junior. The core array from the HOLQ (junior) is ancillary with respect to the covariance parameters.

We also used the  HOLQ to develop a higher order generalization of the SVD. Our version of the SVD can be viewed as a core rotation for the HOSVD (full rank case) or the HOOI (low rank case), where the core is rotated so that the mode specific singular values may be separated from the core array. We note that one can consider the model of \cite{hoff2013equivariant} as a model based truncated ISVD. He considered the model
\begin{align*}
&X \sim N_{p_1\times\cdots\times p_K}((U_1,\ldots,U_K,I_n) \cdot [(D_1,\ldots,D_K,I_n) \cdot V],\sigma^2I_p), \text{ where:}\\
&U_k \text{ is uniformly distributed on } \mathcal{V}_{r_k,p_k},\\
&D_k \text{ has trace 1 and is uniformly distributed on the } r_k \text{ simplex},\\
&V \sim N_{r_1\times\cdots\times r_K}(0,\tau^2I_r), \text{ and}\\
&\tau^2 \sim \text{inverse-gamma}(1/2,\tau_0^2/2),
\end{align*}
where we changed the notation from his paper to make more clear the connection to the ISVD. In such a model, the core $V$ is scaled all-orthonormal in expectation. That is, $E[V_{(k)}V_{(k)}^T] \propto I_{p_k}$ for all $k=1,\ldots,K$. One could extend his results by selecting a prior that allows for non-zero mass for the $D_k$ to be of low rank, as in \cite{hoff2007model} for his model based SVD.

A clear limitation to the utility or the HOLQ or ISVD in practice is that in some dimensions they may not exist, and in other dimensions where they do exist, they may not be unique. The necessary and sufficient conditions for the existence and uniqueness of the  HOLQ are not known. Sufficient conditions for existence and uniqueness occur when $n$ is large. When $n \geq p$, the criterion function, $|| (L_1^{-1},\ldots,L_K^{-1},I_n) \cdot X||$, is bounded below by the value at the LQ decomposition. For $n$ large enough, the  HOLQ is also unique, this follows from the uniqueness of the MLE from \cite{ohlson2013multilinear}. These conditions are equivalently sufficient for the existence and uniqueness of the ISVD. However, in the authors' experiences, the  HOLQ exists and is unique for many dimensions where $n < p$, indeed for many dimensions where $n = 1$. In cases where the HOLQ/ISVD do not exist, the model of \cite{hoff2013equivariant} would be a good alternative. One could also construct a regularized version of the  HOLQ.

We note, however, that when a \emph{local} minimum is reached, then the  HOLQ exists. This is due to the geodesic convexity results of the log-likelihood in \cite{wiesel2012geodesic,wiesel2012convexity}. That is, any local minimum is also a global minimum. These results indicate that, for any particular data set, we can determine if any global minima exist.


\appendix
\section{Proofs}
\subsection{Equivalence of Algorithms \ref{algorithm:flip.flop} and \ref{algorithm:holq}}
In Algorithm \ref{algorithm:flip.flop}, the core array is $Y = (L_1^{-1},\ldots,L_K^{-1},I_n) \cdot X$. Let $MZ$ be the LQ decomposition of $L_kY_{(k)} = X_{(k)}L_{-k}^{-T}$. Then Algorithm \ref{algorithm:flip.flop} updates the core array by
\begin{align*}
Y_{(k)} &\leftarrow (M/|M|^{1/p_k})^{-1}X_{(k)}L_{-k}^{-T}\\
&= (M/|M|^{1/p_k})^{-1}L_kL_k^{-1}X_{(k)}L_{-k}^{-T}\\
&=  (M/|M|^{1/p_k})^{-1}L_kY_{(k)}\\
&= c_1M^{-1}L_kY_{(k)},
\end{align*}
for $c_1 = |M|^{1/p_k}$. Note that $Y_{(k)}/||Y_{(k)}|| = c_2L_{k}^{-1}X_{(k)}L_{-k}^{-T} = c_2L_{k}^{-1}MZ$, for $c_2 = ||Y_{(k)}||^{-1}$. That is, set $L = c_2L_{k}^{-1}M$ so that $LZ$ is the LQ decomposition of $Y_{(k)}/||Y_{(k)}||$. Then Algorithm \ref{algorithm:holq} updates the core array by
\begin{align*}
Q_{(k)} = Y_{(k)}/||Y_{(k)}|| &\leftarrow Z/||Z||\\
&\propto L^{-1}LZ\\
&\propto (L_{k}^{-1}M)^{-1}Y_{(k)} \\
&\propto M^{-1}L_{k}Y_{(k)},
\end{align*}
 Hence, each update of the core array is the same for both algorithms at each iteration up to a scale difference. For each iteration of Algorithm \ref{algorithm:flip.flop}, $L_k \leftarrow M/|M|^{1/p_k}$. But for each iteration of Algorithm \ref{algorithm:holq}, $L_k \leftarrow L_{k}L/|L|^{1/p_k} = L_{k}c_2L_{k}^{-1}M/|c_2L_{k}^{-1}M|^{1/p_k} = M/|M|^{1/p_k}$. Hence, $L_k$ is being updated the same in both algorithms at each iteration.

In this paragraph, we prove that we are updating the scale in Algorithm \ref{algorithm:holq} correctly. Noting that $M = ||Y||L_kL = \ell L_kL$, for $\ell$ the current value of the scale, the update for the scale is 
\begin{align*}
\tilde{\ell} &= ||(M/|M|^{1/p_k})^{-1}X_{(k)}L_{-k}^{-T}||\\
 &= ||(M/|M|^{1/p_k})^{-1}MZ||\\
 &= \left|(||Y||L_kL)\right|^{1/p_k}||Z||\\
 &= \ell|L|^{1/p_K}||Z||.
\end{align*}

\subsection{Update of $k \in J_2$ in HOLQ junior}
For $X \in \mathbb{R}^{p \times n}$, consider finding the minimizer in $D \in \mathcal{D}_{p}^+$ of $||D^{-1}X||$. Using Lagrange multipliers, and letting $S = XX^T$, this is equivalent to minimizing in $D$
\begin{align*}
&\text{tr}(D^{-2}S) + \lambda(|D| - 1) = \sum_{i = 1}^{p}D_{[i,i]}^{-2}S_{[i,i]} + \lambda(\prod_{i=1}^{p}D_{[i,i]} - 1),
\end{align*}
for $D$ a diagonal $p$ by $p$ matrix with positive diagonal elements. The solution to this optimization problem is
\begin{align*}
\tilde{D}_{[i,i]} = \left(S_{[i,i]} / \prod_{i=1}^{p}S_{[i,i]}^{1/p}\right)^{-1/2} \text{ for } i = 1,\ldots,p
\end{align*}
or
\begin{align*}
\tilde{D} = \diag(S_{[1,1]},\ldots,S_{[p,p]})^{-1/2} / |\diag(S_{[1,1]},\ldots,S_{[p,p]})^{-1/2}|^{1/p}.
\end{align*}
So for the block coordinate descent algorithm, for step $k \in J_2$, 
\begin{align}
\begin{split}
\label{equation:diag.1}
&\text{Set } S_k = X_{(k)}L_{-k}^{-T}L_{-k}^{-1}X_{(k)}^T\\
&\text{Set } E = \diag(S_{k[1,1]},\ldots,S_{k[p,p]})^{-1/2}\\
&\text{Set } L_k \leftarrow  E / |E|^{1/p}.
\end{split}
\end{align}
This block coordinate descent algorithm is equivalent to the following steps of simultaneously updating the core array along with the component matrix:
\begin{align}
\begin{split}
\label{equation:diag.2}
 &\text{Set } R_k = Q_{(k)}Q_{(k)}^T\\
 &\text{Set } F = \diag(R_{k[1,1]},\ldots,R_{k[p_k,p_k]})^{1/2}\\
 &\text{Set } \ell \leftarrow \ell|F|^{1/p_k}||F^{-1}Q_{(k)}||\\
 &\text{Set } L_k \leftarrow L_kF/|F|^{1/p_k}\\
 &\text{Set } Q_{(k)} \leftarrow F^{-1}Q_{(k)}/||F^{-1}Q_{(k)}||.
\end{split}
\end{align}
We'll now prove the equivalence of using step (\ref{equation:diag.1}) or step (\ref{equation:diag.2}) to find the HOLQ junior. At each step of (\ref{equation:diag.1}), the core array is $Y = (L_1^{-1},\ldots,L_K^{-1})\cdot X$. Hence, the core array is updated at each iteration of $k \in J_2$ by
\begin{align*}
Y_{(k)} &\leftarrow (E/|E|^{1/p_k})^{-1}X_{(k)}L_{-k}^{-T} = |E|^{1/p_k}E^{-1}X_{(k)}L_{-k}^{-T}.
\end{align*}
Note that, since $Q = Y/||Y||$, we have
\begin{align*}
FF &\propto \diag(Y_{(k)}Y_{(K)}^T) \propto \diag(L_k^{-1}S_kL_k^{-1}) \propto L_k^{-1}EEL_k^{-1}.
\end{align*}
This implies that $F = c_2L_k^{-1}E$ for some constant $c_2$. Hence, the core in (\ref{equation:diag.2}) is being updated by:
\begin{align*}
Q_{(k)} = Y_{(k)}/||Y|| &\leftarrow F^{-1}Q_{(k)} \propto F^{-1}Y_{(k)} = F^{-1}L_k^{-1}X_{(k)}L_{-k}^{-T}\\
 &\propto E^{-1}L_kL_k^{-1}X_{(k)}L_{-k}^{-T} = E^{-1}X_{(k)}L_{-k}^{-T}.
\end{align*}
So the core array is being updated the same at each step, up to a scale difference. Likewise, in (\ref{equation:diag.1}), we have $L_k \leftarrow E/|E|^{1/p_k}$ whereas in (\ref{equation:diag.2}) we have $L_k \leftarrow L_kF/|F|^{1/p_k} = L_kc_2L_k^{-1}E/|c_2L_k^{-1}E|^{1/p_k} = E/|E|^{1/p_k}$, so each component matrix is being updated the same at each iteration.

Note that the diagonal elements of $Q_{(k)}Q_{(k)}^T$ in (\ref{equation:diag.2}) are being scaled to be $1/p_k$ at each iteration. Hence, any fixed point of this algorithm must have the property that $\diag\left(Q_{(k)}Q_{(k)}^T\right) = \mathbf{1}_{p_k}/p_k$ for all $k \in J_2$. In other words,the rows of $Q_{(k)}$ have Frobenius norm $1/p_k$.

\subsection{Update of $k \in J_3$ in HOLQ junior}
For $K=1$ and $n \geq p$, we require finding the $L_k \in \mathcal{G}_{p_k}^{Ch}$ that minimizes $||L^{-1}X||^2$. Using Lagrange multipliers, this is equivalent to finding the $V$ in the general linear group of $p$ by $p$ non-singular matrices that minimizes
\begin{align}
\label{equation:constrain.ldu}
\text{tr}(VV^TXX^T) - \text{tr}(\Lambda_1(V - I_p)) - \text{tr}(\mathbf{1}_{p}\mathbf{1}_{p}^T(\Lambda_2 * V)),
\end{align}
where $\Lambda_1 = \diag(\lambda_1,\ldots,\lambda_p)$, $\mathbf{1}_{p}$ is the $p$-dimensional vector of $1$'s, ``$*$'' is the Hadamard (element-wise) product, and $\Lambda_2$ is upper triangular with $0$'s in the diagonal. That is,
\begin{align*}
&\Lambda_2 = 
  \left(
    \begin{array}{cccccc}
      0 & \lambda_{1,2} & \lambda_{1,3} & \cdots & \cdots & \lambda_{1,p}\\
      0 & 0 & \lambda_{2,3} & \cdots & \cdots & \lambda_{2,p}\\
      \vdots & & \ddots & & & \vdots\\
      \vdots & & & \ddots & & \vdots\\
      \vdots & & & & 0 & \lambda_{p-1,p}\\
      0 & \cdots & \cdots & \cdots & \cdots & 0
    \end{array}
  \right).
\end{align*}
The idea is that the Lagrange multipliers in $\Lambda_1$ are constraining the diagonal elements of $V$ to be $1$, and the Lagrange multipliers in $\Lambda_2$ are constraining the upper triangular elements of $V$ to be $0$. Once we find the minimizer, we can set $L = V^{-1}$. Taking derivatives of (\ref{equation:constrain.ldu}) and setting equal to zero, we have
\begin{align*}
&2XX^TV - \Lambda_1 - \Lambda_2 = 0 \text{ and } V \in \mathcal{G}_{p}^{Ch}\\
&\Leftrightarrow V = (\Lambda_1 + \Lambda_2)(XX^T)^{-1}/2 \text{ and } V \in \mathcal{G}_{p}^{Ch}.
\end{align*}
Note that $\Lambda_1 + \Lambda_2$ is upper triangular. By the uniqueness of the LDU decomposition of $XX^T = U^TDU$ where $U^T \in \mathcal{G}_{p_k}^{Ch}$ \citep[Corollary 3.5.5]{horn2012matrix}, the only critical point occurs at $V = U^{-T}$ and $\Lambda_1 + \Lambda_2 = 2DU$.

Since the constraints are all linear, to prove that this minimizer is a global minimizer, it suffices to prove that $\text{tr}(VV^TXX^T)$ is convex in $V$. But by Exercise 1 of Section 10.6 in \cite{magnus1999matrix}, the Hessian matrix is $2XX^T \otimes I_p$, which is clearly positive definite.

To summarize, the minimizer of $||\tilde{L}^{-1}X||^2$ in $\tilde{L} \in \mathcal{G}_{p}^{Ch}$ is $U^T$ from the LDU decomposition of $XX^T = U^TDU$. This is equivalent to taking the LQ decomposition of $X = LQ$, then setting $F = \diag(L_{[1,1]},\ldots,L_{[p,p]})$. The minimizer is then $LF^{-1}$. What's ``left over'' after multiplying out $LF^{-1}$ is then $FQ$, which has orthogonal (though not necessarily orthonormal) rows. 

For modes where $L_k \in \mathcal{G}_{p_k}^{Ch}$, we thus update $L_k$ and the core $Q$ by:
\begin{align*}
&\text{Take LQ decomposition of core }Q_{(k)} = LZ\\
&\text{Set } F = \diag(L_{[1,1]},\ldots,L_{[p,p]})\\
&L_k \leftarrow L_kLF^{-1}\\
&Q_{(k)} \leftarrow FZ / ||FZ||\\
&\ell \leftarrow \ell||FZ||.
\end{align*}
Hence, any fixed point, including the HOLQ junior, must have the property that $Q_{(k)}$ has orthogonal, though not necessarily orthonormal, rows. This proves the last part of Theorem \ref{theorem:holq.junior.core}.

\subsection{Proof of Theorem \ref{theorem:lrt.dist}}
We can represent models $H_0$ and $H_1$ in Theorem \ref{theorem:lrt.stat} as being generated under two different groups. Let $\{1,\ldots,K\} = J_1 \cup J_2 \cup J_3 \cup J_4$. Let $\Psi_k \in \mathcal{G}_0^{(k)}$ where $\mathcal{G}_0^{(k)} = \mathcal{G}_{p_k}^+$, $\mathcal{D}_{p_k}^+$, $\mathcal{G}_{p_k}^{Ch}$, or $\{I_{p_k}\}$ if $k$ is in $J_1$, $J_2$, $J_3$, or $J_4$, respectively. Let $\{1,\ldots,M\} = \tilde{J}_1 \cup \tilde{J}_2 \cup \tilde{J}_3 \cup \tilde{J}_4$. Let $\Phi_m \in \mathcal{G}_1^{(m)}$ and $\mathcal{G}_1^{(m)} = \mathcal{G}_{q_m}^+$, $\mathcal{D}_{q_m}^+$, $\mathcal{G}_{p_k}^{Ch}$, or $\{I_{q_m}\}$ if $m$ is in $\tilde{J}_1$, $\tilde{J}_2$, $\tilde{J}_3$, or $\tilde{J}_4$, respectively. Then the models of $H_0$ and $H_1$ can be represented by
\begin{align*}
&H_0: \veco(X) \overset{d}{=} \sigma_0(\Psi_1,\ldots,\Psi_K)\cdot\veco(Z)\\
&H_1: \veco(X) \overset{d}{=} \sigma_1(\Phi_1,\ldots,\Phi_M)\cdot\veco(Z),
\end{align*}
where $\veco(Z)$ is a $p$ vector with standard normal entries. Saying that $H_0$ is a submodel of $H_1$ is equivalent to saying that $\mathcal{G}_0$ is a subgroup of $\mathcal{G}_1$. Hence, we are in the situation of having a hypothesis testing problem that is invariant under $\mathcal{G}_0$ \citep[Definition 3.2]{eaton1989group}. The LRT statistic is an invariant function \citep[Proposition 7.13]{eaton1983multivariate}. The distribution of any invariant function depends only on the maximally invariant parameter \citep[Theorem 6.3.2]{lehmann2006testing}. Under the null, the maximally invariant parameter is a constant because the group action is transitive over the parameter space (since the model is generated by $\mathcal{G}_0$).

\subsection{Proof of equivalence of Algorithms \ref{algorithm:ihop} and \ref{algorithm:pancake}.}

The core array from Algorithms \ref{algorithm:ihop} and \ref{algorithm:pancake} is $R \propto (L_1^{-1},\ldots,L_K^{-1},I_n)\cdot X$. Let $LL^TZ$ be the polar decomposition of $X_{(k)}L_{-k}^{-1}$ for $L$ lower triangular with positive diagonal elements and $Z^T\in\mathcal{V}_{p_k,np/p_k}$. This is, equivalently, the polar decomposition of $L_kR_{(k)}$ as in Algorithm \ref{algorithm:pancake}. Thus we have
\begin{align*}
R &\leftarrow cL^{-1}X_{(k)}L_{-k}^{-1}\\
& \propto L^{-1}LL^TZ\\
& \propto L^TZ.
\end{align*}
Hence, we are updating the core array correctly. $L_k$ is trivially being updated correctly in the Algorithm \ref{algorithm:pancake}. To see that we are updating the scale correctly, note that
\begin{align*}
\ell &= ||(L/||L||)^{-1}X_{(k)}L_{-k}^{-1}||\\
&= ||L||||L^{-1}LL^TZ||\\
&= ||L||||L^TZ||\\
&= ||L||||R||.
\end{align*}

\subsection{Theorem 1 from \cite{anderson1986maximum}}
\label{section:appendix.anderson}
The following is a simplified version of Theorem 1 from \cite{anderson1986maximum}.
\begin{theorem}
Let $\Omega$ be a set in the space of $\mathcal{S}_{p}^+$ such that if $S \in \Omega$ then $cS\in \Omega$ for all $c > 0$. For $X \in \mathbb{R}^{n \times p}$, suppose that $g$ is such that $g(\text{tr}(XX^T))$ is a density in $\mathbb{R}^{n\times p}$ and $y^{np/2}g(y)$ has a finite positive maximum at $y_g$. Suppose on the basis of an observation of $X$ from $|\Sigma|^{-n/2}g(\text{tr}(X\Sigma^{-1}X^T))$, the MLE under normality $\hat{\Sigma} \in \Omega$ exists and is unique and that $\hat{\Sigma}$ is positive definite with probability 1. Then the MLE for $g$ is proportional to $\hat{\Sigma}$ and the maximum of the likelihood is $|\hat{\Sigma}|^{-n/2}g(y_g)$.
\end{theorem}
In this paper, $\Omega$ is the cone of Kronecker structured covariance matrices. This result says that the MLE under elliptically contoured distributions is proportional to the MLE under normality. In our paper, we use the parameterization where $|\sigma^2\Sigma_K\otimes\cdots\otimes\Sigma_1| = (\sigma^2)^p$. Hence, $|\hat{\Sigma}|^{-n/2}g(y_g) = \hat{\sigma}^{-np} g(y_g)$. For the HOLQ junior, we are implying that $n = 1$, with the understanding that some modes might be the identity.
\bibliography{testing_bib}

\end{document}